\documentclass{scrartcl}
\usepackage{color}
\usepackage{graphicx}
\addtokomafont{sectioning}{\sc}
\usepackage[utf8]{inputenc}
\usepackage[T1]{fontenc}
\usepackage[sc,osf]{mathpazo}
\usepackage[euler-digits,euler-hat-accent]{eulervm}
\usepackage[leqno]{amsmath}
\usepackage{amssymb}
\usepackage{amsthm}
\usepackage[noend]{algorithmic}
\usepackage{hyperref}

\newtheorem{theorem}{Theorem}[section]
\newtheorem{lemma}[theorem]{Lemma}
\newtheorem{corollary}[theorem]{Corollary}
\newtheorem{proposition}[theorem]{Proposition}

\theoremstyle{remark}
\newtheorem{remark}{Remark}
\newtheorem{example}{Example}
\newtheorem{algorithm}{Algorithm}{\bf}{\it}

\def\xdag{x^\dagger}
\def\ydel{y^\delta}
\def\R{\mathbb{R}}
\def\N{\mathbb{N}}

\begin{document}
\title{Computing quasisolutions of nonlinear inverse problems via efficient minimization of trust region problems}
\author{Barbara Kaltenbacher and Franz Rendl and Elena Resmerita\\
Institute of Mathematics, Alpen-Adria-Universit\"at Klagenfurt}
\maketitle
\begin{abstract}
In this paper we present a method for the regularized solution of nonlinear inverse problems, based on Ivanov regularization (also called method of quasi solutions or constrained least squares regularization). This leads to the minimization of a non-convex cost function under a norm constraint, where non-convexity is caused by nonlinearity of the inverse problem. Minimization is done by iterative approximation, using (non-convex) quadratic Taylor expansions of the cost function. This leads to repeated solution of quadratic trust region subproblems with possibly indefinite Hessian. Thus the key step of the method consists in application of an efficient method for solving such quadratic subproblems, developed by Rendl and Wolkowicz \cite{RendlWolkowicz97}. We here present a convergence analysis of the overall method as well as numerical experiments. 
\end{abstract}
\section{Introduction}
Consider the nonlinear inverse problem of recovering $x\in X$ in 
\begin{equation}\label{Fxy}
F(x)=y
\end{equation}
-- with a forward operator $F:\mathcal{D}(F)\subseteq X\to Y$ mapping from a Hilbert space $X$ to a Banach space $Y$ -- from noisy measurements $\ydel$ of $y\in Y$ satisfying 
\begin{equation}\label{delta}
\mathcal{S}(y,\ydel)\leq\delta
\end{equation}
where $\mathcal{S}:y^2\to\R^+$ is a distance measure quantifying the data misfit, for instance a norm but possibly also a more involved expression such as the Kullback Leibler divergence arising from certain statistic measurement noise models. 

The regularized solution of \eqref{Fxy} via the method of quasi solutions (also called Ivanov regularization) leads to minimization problems of the form
\begin{equation}\label{Ivanov}
x_\rho^\delta\in\mbox{argmin}\{ r^\delta(x)\, : \, x\in X\,, \ \|x\|^2\leq \rho\}\,,
\end{equation}
where 
\[
r^\delta(x)=\mathcal{S}(F(x),\ydel)
\]
with an appropriately chosen radius $\rho$ that plays the role of a regularization parameter here, cf., 
\cite{GrodzevichWolkowicz09, Ivanov62, Ivanov63, IvanovVasinTanana02, LorenzWorliczek13, NeubauerRamlau14, SeidmanVogel89,Vogel90}.
The selection of $\rho$ can be done in an a priori fashion if the norm of some exact solution $\xdag$ to \eqref{Fxy} is known. Namely, in that case the choice $\rho=\|\xdag\|^2$ can shown to be optimal. Otherwise, also a noise level dependent a posteriori choice of $\delta$ according to Morozov's discrepancy principle, i.e., $\rho=\rho(\delta)$ such that
\begin{equation}\label{discrprinc}
\delta < r^\delta(x_\rho) \leq \tau\delta
\end{equation}
for some $\tau>1$ fixed independent of the noise level $\delta$, leads to convergence, cf. \cite{CKK15}.

In this paper we wish to exploit the obvious relation to trust region subproblems to take advantage of an efficient algorithm proposed in \cite{RendlWolkowicz97} for solving quadratic trust region subproblems with not necessarily positive semidefinite Hessian -- a situation that is highly relevant here in view of the fact that the cost functional in \eqref{Ivanov} exhibits potential nonconvexity due to nonlinearity of the forward operator $F$ and/or the distance measure $\mathcal{S}$.

To this end, we first of all discretize the problem by restriction of minimization to finite dimensional subspaces $X_n\subseteq X$ and the use of computational approximations of the involved operators and norms, which leads to a sequence of finite dimensional problems
\begin{equation}\label{Ivanov_n}
x_{n,\rho}^\delta\in\mbox{argmin}\{ r_n^\delta(x_n)\, : \, x_n\in X_n\,, \ \|x_n\|_n^2\leq \rho\}
\end{equation}
where 
\[
r_n^\delta(x_n)=\mathcal{S}_n(F_n(x_n),\ydel)\,,
\]
where $\mathcal{S}_n:Y^2\to\R$  $F_n:X_n\to Y$ are approximations, e.g., due to discratization, of $\mathcal{S}$ and $F$, respectively.

We also consider the practically relevant situation that these minimization problems are not solved with infinite precision but in an inexact sense, i.e., we will use regularized approximations $\hat{x}_{n,\rho}^\delta\in X_n$ satisfying
\begin{equation}\label{Ivanov_n_eta}
\|\hat{x}_{n,\rho}^\delta\|_n^2\leq \rho\mbox{ and }
r_n^\delta(\hat{x}_{n,\rho}^\delta)
\leq \min_{x_n\in X_n} \{r_n^\delta(x_n)\ : \ \|x_n\|_n^2\leq \rho\}+\eta_n^\delta
\end{equation}
and choose $\rho=\rho(\delta)$ such that 
\begin{equation}\label{discrprinc_n_eta}
\delta +\underline{\eta}_n^\delta < r_n^\delta(\hat{x}_{n,\rho}^\delta) \leq \tau\delta +\overline{\eta}_n^\delta
\end{equation}
with appropriately chosen tolerances $\eta_n^\delta$, $\underline{\eta}_n^\delta$, $\overline{\eta}_n^\delta$.
A first requirement on these tolerances to admit solutions to \eqref{Ivanov_n_eta} and \eqref{discrprinc_n_eta} is
\begin{equation}\label{eta1}
\eta_n^\delta\geq0 \mbox{ and } \underline{\eta}_n^\delta < (\tau-1)\delta +\overline{\eta}_n^\delta
\end{equation}

We mention in passing that this nonconvex case is also of special interest for Ivanov regularization since this is also the situation in which it is in general not equivalent to Tikhonov regularization.

\section{Regularizing property of $\hat{x}_{n,\rho}^\delta$}\label{sec:reg}
To analyze convergence of $\hat{x}_{n,\rho}^\delta$ to an exact solution $\xdag$ of \eqref{Fxy} as $\delta\to0$, we first of all state a straightforward monotonicity property for the exact minimizers at fixed discretization level $n$.

\begin{lemma}\label{lem:mon}
The mapping $\rho\mapsto r_n^\delta(x_{n,\rho}^\delta)$ with $x_{n,\rho}^\delta$ according to \eqref{Ivanov_n} is monotonically decreasing.
\end{lemma}
\begin{proof}
Since the admissible set is larger for larger radius and we minimize the same cost function, the assertion is obvious.
\end{proof}

Moroever, we get a uniform bound on the radii chosen according to \eqref{discrprinc_n_eta} with $\hat{x}_{n,\rho}^\delta$ satisfying \eqref{Ivanov_n_eta}, provided the tolerances are chosen appropriately.
\begin{lemma}\label{lem:boundrho}
Let \eqref{eta1} and 
\begin{equation}\label{eta2}
\underline{\eta}_n^\delta -\eta_n^\delta\geq r_n^\delta(P_n\xdag)-r^\delta(\xdag) 
\end{equation}
hold, where $P_n$ is the orthogonal projection onto $X_n$.\\ 
Then $\rho(\delta)$ according to \eqref{discrprinc_n_eta} satisfies the estimate
\begin{equation}\label{eq:boundrho}
\rho(\delta)\leq \|P_n \xdag\|_n^2=:\rho_n^\dagger
\end{equation}
\end{lemma}
\begin{proof}
By \eqref{Ivanov_n_eta} and \eqref{discrprinc_n_eta}, as well as \eqref{delta} we have, for $\rho=\rho(\delta)$,
\[
r_n^\delta(x_{n,\rho}^\delta)
\geq r_n^\delta(\hat{x}_{n,\rho}^\delta)-\eta_n^\delta
>\delta+\underline{\eta}_n^\delta-\eta_n^\delta
\geq r^\delta(\xdag)+\underline{\eta}_n^\delta-\eta_n^\delta
\geq r_n^\delta(P_n\xdag)
\geq r_n^\delta(x_{n,\rho_n^\dagger}^\delta)\,,
\]
where in the last two inequalities we have used \eqref{eta2} and minimality of $x_{n,\rho_n^\dagger}^\delta$ (together with feasibility of $P_n\xdag$ for the discretized problem). Thus the assertion follows by contraposition in Lemma \ref{lem:mon}
\end{proof}

To prove convergence and convergence rates, we will make use of two general results Propositions \ref{prop:conv} and \ref{prop:rates}, see also Theorems 2.4 and 2.7 in  \cite{CKK15} for a slightly different setting.
Here for some fixed maximal noise level $\bar{\delta}$, $(y^\delta)_{\delta\in(0,\bar{\delta}]}$ is a familiy of data satisfying \eqref{delta} and $(\tilde{x}^\delta)_{\delta\in(0,\bar{\delta}]}$ a family of regularized approximations (defined by \eqref{Ivanov_n_eta} with \eqref{discrprinc_n_eta} or by some other regularization method).

\begin{proposition}\label{prop:conv}
Let $F$ be weakly sequentially closed at $y$ in the sense that 
\begin{equation}\label{weakseqclosed}
\begin{aligned}
&\forall (y_k)_{k\in\N}, \ (x_k)_{k\in\N} : \quad \\
&(x_k\rightharpoonup x\mbox{ and }\mathcal{S}(F(x_k),y_k)\to0\mbox{ and } \mathcal{S}(y,y_k)\to0 \ \Rightarrow \ y\in\mathcal{D}(F) \mbox{ and } F(x)=y
\end{aligned}
\end{equation}
and let uniform boundedness
\begin{equation}\label{boundxtilde}
\exists \bar{\rho}>0\ \forall \delta\in(0,\bar{\delta}]: \quad \|\tilde{x}^\delta\|^2\leq\bar{\rho}\end{equation}
as well as convergence of the operator values in the sense of the distance measure $\mathcal{S}$
\begin{equation}\label{Fconv}
\mathcal{S}(F(\tilde{x}^\delta),\ydel)\to0\mbox{ as }\delta\to0
\end{equation}
hold.\\
Then there exists a weakly convergent subsequence of $(\tilde{x}^\delta)_{\delta\in(0,\bar{\delta}]}$ and the limit $x^*$ of every weakly convergent subsequence of $(\tilde{x}^\delta)_{\delta\in(0,\bar{\delta}]}$ satisfies \eqref{Fxy}.
\\
If for such a weakly convergent subsequence $(\tilde{x}^{\delta_k})_{k\in\N}$ with limit $x^*$, 
\eqref{boundxtilde} can be strengthened to 
\begin{equation}\label{limsupxtilde}
\limsup_{k\to\infty} \|\tilde{x}^{\delta_k}\|^2\leq\|x^*\|^2\,,
\end{equation}
then we even have strong convergence $\tilde{x}^{\delta_k}\to x^*$ as $k\to\infty$.
\\
If the solution $\xdag$ to \eqref{Fxy} is unique then $\tilde{x}^\delta\rightharpoonup\xdag$ as $\delta\to0$ under condition\eqref{boundxtilde} and $\tilde{x}^\delta\to\xdag$ as $\delta\to0$ under condition $\limsup_{\delta\to0} \|\tilde{x}^{\delta}\|^2\leq\|\xdag\|^2$.
\end{proposition}

To state convergence with rates to some solution $\xdag$ of \eqref{Fxy} we make use of a variational source condition
\begin{equation}\label{vsc}
\exists \beta\in(0,1)\ \forall x\in\mathcal{B}_R(\xdag) \quad 2(\xdag, \xdag-x)\leq
\beta\|\xdag-x\|^2+\varphi(\mathcal{S}(F(\xdag),F(x)))
\end{equation}
with some radius $R>0$ and some index function $\varphi:\R^+\to\R^+$, (i.e. $\varphi$ is monotonically increasing and $\lim_{t\to0}\varphi(t)=0$).
Condition \eqref{vsc} is a condition on the smoothness of $\xdag$ that is the stronger the faster $\varphi$ decays to zero as $t\to0$. It is, e.g., satisfied with $\mathcal{S}(y_1,y_2)=\frac12\|y_1-y_2\|^2$ and  $\varphi(t)\sim t$ if $\xdag$ lies in the range of the adjoint of the linearized forward operator (which is typically a smoothing operator) and $F'$ is Lipschitz continuous
\[
\xdag =F'(\xdag)^*w\mbox{ and }\forall x\in\mathcal{B}_R(\xdag)\,: \ \|F'(x)-F'(\tilde{x})\|\leq L
\|x-\tilde{x}\| \mbox{ and }L\|w\|<1\,.
\]
By a homogeneity argument for the case of linear $F$ it can be seen that the fastest possible decay of $\varphi$ at zero that gives a reasonable assumption in $\eqref{vsc}$ is $\varphi(t)\sim t$.

\begin{proposition}\label{prop:rates}
Let \eqref{vsc} hold and $\tilde{x}^\delta$ be contained in $\mathcal{B}_R(\xdag)$ for all $\delta\in(0,\bar{\delta}]$. Moreover, assume that there exist constants $C_1,C_2,C_3>0$ such that 
\begin{equation}\label{boundxtilderate}
\forall \delta\in(0,\bar{\delta}]: \quad \|\tilde{x}^\delta\|^2\leq
\|\xdag\|^2+C_1\varphi(C_2\delta)
\mbox{ and }\mathcal{S}(F(\tilde{x}^\delta),\ydel)\leq C_3 \delta\,.
\end{equation}
and that $\mathcal{S}$ satisfies the generalized triangle inequality
\begin{equation}\label{triang}
\mathcal{S}(y_1,y_2)\leq C_4 (\mathcal{S}(y_1,y_3)+\mathcal{S}(y_2,y_3))
\end{equation}
for some $C_4>0$ and all $y_1,y_2\in Y$.

Then the convergence rate 
\[
\|\tilde{x}^\delta-\xdag\|^2=O(\varphi(C\,\delta)) \mbox{ as }\delta\to0
\]
holds with $C=\max\{C_2,C_4(C_3+1)\}$.
\end{proposition}
\begin{proof}
From \eqref{vsc} with $x=\tilde{x}^\delta$ and \eqref{boundxtilderate}, as well as \eqref{triang}, \eqref{delta}, and monotonicity of $\varphi$ we get 
\[
\begin{aligned}
&\|\tilde{x}^\delta-\xdag\|^2 
= \|\tilde{x}^\delta\|^2-\|\xdag\|^2+2(\xdag,\xdag-\tilde{x}^\delta)\\
&\leq \beta \|\tilde{x}^\delta-\xdag\|^2+C_1\varphi(C_2\delta)+\varphi(C_4(\mathcal{S}(F(\tilde{x}^\delta),\ydel)+\delta))\,.
\end{aligned}
\]
\end{proof}

\begin{corollary}
Let \eqref{weakseqclosed} hold and let $\hat{x}_{n(\delta),\rho(\delta)}^\delta$ be defined by \eqref{Ivanov_n_eta}, \eqref{discrprinc_n_eta}, with the discretization level $n=n(\delta)$ and the tolerances 
$\eta_{n(\delta)}^\delta$
$\underline{\eta}_{n(\delta)}^\delta$
$\overline{\eta}_{n(\delta)}^\delta$
chosen such that \eqref{eta1}, \eqref{eta2} and 
\begin{equation}\label{eta3}
\begin{aligned}
&\|P_{n(\delta)}x^\dagger\|_{n(\delta)}^2-\|\xdag\|^2\leq C_{1,1} \varphi(C_2\delta)\\
&\left.\begin{array}{r}\|x\|^2-\|x\|_{n(\delta)}^2\leq C_{1,2} \varphi(C_2\delta)\\[0.5ex] 
						r^\delta(x)-r_{n(\delta)}^\delta(x)\leq \tau_1\delta
\end{array}\right\}
\  \mbox{ for }x=\hat{x}_{n(\delta),\rho(\delta)}^\delta\\
&\overline{\eta}_{n(\delta)}^\delta\leq\tau_2\delta
\end{aligned}
\end{equation}
hold for fixed constants $\tau_1,\tau_2,C_{1,1},C_{1,2},C_2>0$ independent of $\delta$.\\
Then $\hat{x}_{n(\delta),\rho(\delta)}^\delta$ converges to $\xdag$ subsequentially in the sense of Proposition \ref{prop:conv}. \\
If additionally a variational source condition \eqref{vsc} and the generalized triangle inequality \eqref{triang} holds, then  
\[
\|\hat{x}_{n(\delta),\rho(\delta)}^\delta-\xdag\|^2=O(\varphi(C\delta)) \mbox{ as }\delta\to0\,,
\]
with $C=\max\{C_2,C_4(\tau+\tau_1+\tau_2+1)\}$.
\end{corollary}

\begin{proof}
By Lemma \ref{lem:boundrho} and \eqref{eta3} we have 
\[
\|\hat{x}_{n(\delta),\rho(\delta)}^\delta\|^2\leq \|\xdag\|^2+(C_{1,1}+C_{1,2})\varphi(C_2\delta)\,.
\]
Moreover, by \eqref{discrprinc_n_eta} and \eqref{eta3} we can estimate
\[
\mathcal{S}(F(\hat{x}_{n(\delta),\rho(\delta)}^\delta),\ydel)=r^\delta(\hat{x}_{n(\delta),\rho(\delta)}^\delta)
\leq r_n^\delta(\hat{x}_{n(\delta),\rho(\delta)}^\delta)+\tau_1\delta\leq (\tau+\tau_1+\tau_2)\delta\,.
\]
\end{proof}
\section{Computation of $\hat{x}_{n,\rho}^\delta$ and of $\rho(\delta)$}
\subsection{A Newton type iteration for computing $\hat{x}_{n,\rho}^\delta$}
For fixed discretization and noise level $n,\delta$ 
%(which we mostly skip in our notation in this section) 
and fixed radius $\rho<\rho(\delta)$, we approximate the nonlinear trust region subproblem \eqref{Ivanov_n} 
\begin{equation}\label{minr}
\min_{x_n\in X_n} \ r_n^\delta(x_n)\mbox{ s.t. } \|x_n\|_n^2\leq \rho
\end{equation}
by a sequence of quadratic subproblems arising from second order Taylor expansion of the cost function
\begin{equation}\label{minq}
\min_{x_n\in X_n} \ q^k(x_n)\mbox{ s.t. } \|x_n\|_n^2\leq \rho
\end{equation}
with 
\begin{equation}\label{q}
q^k(x)=r_n^\delta(x^k)+{r_n^\delta}'(x^k)(x-x^k)+\frac12 {r_n^\delta}''(x^k)(x-x^k)^2
\end{equation}
where $x^k$ is some current iterate.
Necessary second order optimality conditions for these two minimization problems (in case of \eqref{minq} they will also be sufficient, cf. \cite{Sorensen82}) are existence of $\lambda_\rho,\lambda^{k+1}\in\R$ such that  
\begin{eqnarray}
&&{r_n^\delta}'(x_{n,\rho}^\delta)+\lambda_\rho \langle x_{n,\rho}^\delta,\cdot\rangle_n=0	
\label{eq_rho}\\
&&\lambda_\rho\geq0\,, \ \|x_{n,\rho}^\delta\|_n^2\leq\rho\,, \ \lambda_\rho(\|x_{n,\rho}^\delta\|_n^2-\rho)=0	
\label{ie_rho}\\
&& \forall w\in C(x_{n,\rho}^\delta)\,: \ {r_n^\delta}''(x_{n,\rho}^\delta)w^2\geq 0
\label{sd_rho}
\end{eqnarray}  
%(cf. \cite{Bonnans} for the latter)
for \eqref{minr}, where for these simple constraints the critical cone is given by  
\begin{equation*}
C(x_{n,\rho}^\delta)=\begin{cases} 
X_n \mbox{ if } \|x_{n,\rho}^\delta\|_n^2<\rho\\
\{x_{n,\rho}^\delta\}_{\leq}=\{v\in X_n\, : \, \langle x_{n,\rho}^\delta,v\rangle_n\leq0\} 
\mbox{ if } \|x_{n,\rho}^\delta\|_n^2=\rho\mbox{ and }\lambda_\rho=0\\
\{x_{n,\rho}^\delta\}^\bot=\{v\in X_n\, : \, \langle x_{n,\rho}^\delta,v\rangle_n=0\} 
\mbox{ if } \|x_{n,\rho}^\delta\|_n^2=\rho\mbox{ and }\lambda_\rho>0
\end{cases}
\end{equation*}
and 
\begin{eqnarray}
&&{r_n^\delta}'(x_n^k)+{r_n^\delta}''(x_n^k)(x_n^{k+1}-x_n^k)+\lambda^{k+1}\langle x_n^{k+1},\cdot\rangle_n=0	
\label{eq_kp1}\\
&&\lambda^{k+1}\geq0\,, \ \|x_n^{k+1}\|_n^2\leq\rho\,, \ \lambda^{k+1}(\|x_n^{k+1}\|_n^2-\rho)=0	
\label{ie_kp1}\\
&& \forall w\in X_n\,: \ {r_n^\delta}''(x_n^k)w^2+\lambda^{k+1}\|w\|_n^2\geq 0
\label{sd_kp1}
\end{eqnarray}  
for \eqref{minq} with \eqref{q}.
For the error $x_n^{k+1}-x_{n,\rho}^\delta$ this implies
\begin{equation}\label{err_rho}
\begin{aligned}
&({r_n^\delta}''(x_{n,\rho}^\delta)+\lambda_\rho I_n)(x_n^{k+1}-x_{n,\rho}^\delta)\\
&={r_n^\delta}''(x_{n,\rho}^\delta) (x_n^{k+1}-x_{n,\rho}^\delta)
+\langle-\lambda_\rho x_{n,\rho}^\delta+\lambda^{k+1} x_n^{k+1} +(\lambda_\rho-\lambda^{k+1})x_n^{k+1},\cdot\rangle_n\\
&={r_n^\delta}''(x_{n,\rho}^\delta) (x_n^{k+1}-x_{n,\rho}^\delta)+{r_n^\delta}'(x_{n,\rho}^\delta)
-{r_n^\delta}'(x_n^k)-{r_n^\delta}''(x_n^k)(x_n^{k+1}-x_n^k)
+(\lambda_\rho-\lambda^{k+1})\langle x_n^{k+1},\cdot\rangle_n\\
&=:t_\rho+(\lambda_\rho-\lambda^{k+1})\langle x_n^{k+1},\cdot\rangle_n
\end{aligned}
\end{equation}
where $I_n:X_n^2\to\R$, $I_n(x_n,\tilde{x}_n)=\langle x_n,\tilde{x}_n\rangle_n$, and 
\begin{equation}\label{err_kp1}
\begin{aligned}
&({r_n^\delta}''(x_n^k)+\lambda^{k+1} I_n)(x_n^{k+1}-x_{n,\rho}^\delta)\\
&={r_n^\delta}''(x_n^k) (x_n^{k+1}-x_{n,\rho}^\delta)
+\langle-\lambda_\rho x_{n,\rho}^\delta+\lambda^{k+1} x_n^{k+1} +(\lambda_\rho-\lambda^{k+1})x_{n,\rho}^\delta,\cdot\rangle_n\\
&={r_n^\delta}''(x_n^k) (x_n^k-x_{n,\rho}^\delta)+{r_n^\delta}'(x_{n,\rho}^\delta)
-{r_n^\delta}'(x_n^k)
+(\lambda_\rho-\lambda^{k+1})\langle x_{n,\rho}^\delta, \cdot\rangle_n\\
&=:t^{k+1}
+(\lambda_\rho-\lambda^{k+1})\langle x_{n,\rho}^\delta, \cdot\rangle_n
\end{aligned}
\end{equation}
where under a Lipschitz condition on ${r_n^\delta}''$
\begin{equation}\label{Lipschitz_rpp}
\forall x_n,\tilde{x}_n\in \mathcal{D}(F_n) \,: \ \|{r_n^\delta}''(x)-{r_n^\delta}''(\tilde{x})\|_n\leq L\|x-\tilde{x}\|_n
\end{equation}
we have 
\[
\begin{aligned}
\|t_\rho\|_n=&\|\int_0^1 \left[{r_n^\delta}''(x_{n,\rho}^\delta+\theta(x_{n,\rho}^\delta-x_n^k))
-{r_n^\delta}''(x_{n,\rho}^\delta)\right]\, d\theta (x_{n,\rho}^\delta-x_n^k)\\
&+\left[{r_n^\delta}''(x_{n,\rho}^\delta)-{r_n^\delta}''(x_n^k)\right](x_n^{k+1}-x_n^k)\|_n\\
\leq& \frac{L}{2}\|x_{n,\rho}^\delta-x_n^k\|_n^2+L\|x_{n,\rho}^\delta-x_n^k\|_n\, \|x_n^{k+1}-x_n^k\|_n
\end{aligned}
\]
\[
\begin{aligned}
\|t^{k+1}\|_n=&\|\int_0^1 \left[{r_n^\delta}''(x_{n,\rho}^\delta+\theta(x_{n,\rho}^\delta-x_n^k))
-{r_n^\delta}''(x_{n,\rho}^\delta)\right]\, d\theta (x_{n,\rho}^\delta-x_n^k)\\
\leq& \frac{L}{2}\|x_{n,\rho}^\delta-x_n^k\|_n^2
\end{aligned}
\]
Testing the sum of \eqref{err_rho} and \eqref{err_kp1} with $x_n^{k+1}-x_{n,\rho}^\delta$ and using the fact that 
\[
\begin{aligned}
&(\lambda_\rho-\lambda^{k+1})
\langle x_n^{k+1}+x_{n,\rho}^\delta,x_n^{k+1}-x_{n,\rho}^\delta\rangle_n
=(\lambda_\rho-\lambda^{k+1})(\|x_n^{k+1}\|_n^2-\|x_{n,\rho}^\delta\|_n^2)\\
&=\begin{cases}
\lambda_\rho(\|x_n^{k+1}\|_n^2-\rho)\mbox{ if } \|x_{n,\rho}^\delta\|_n^2=\rho\\
\lambda^{k+1}(\|x_{n,\rho}^\delta\|_n^2-\rho)\mbox{ if } \|x_{n,\rho}^\delta\|_n^2<\rho
\end{cases}
\leq 0
\end{aligned}
\]
(the latter representation is readily checked by a distinction of the cases $\|x_n^{k+1}\|_n^2= / < \rho$)
we end up with 
\begin{equation}\label{est1}
\begin{aligned}
&({r_n^\delta}''(x_{n,\rho}^\delta)+{r_n^\delta}''(x_n^k)+\lambda^{k+1} I_n)(x_n^{k+1}-x_{n,\rho}^\delta)^2 +\lambda_\rho\|x_n^{k+1}-x_{n,\rho}^\delta\|_n^2\\
&\leq L(\|x_{n,\rho}^\delta-x_n^k\|_n^2+\|x_{n,\rho}^\delta-x_n^k\|_n\, \|x_n^{k+1}-x_n^k\|_n)
\|x_n^{k+1}-x_{n,\rho}^\delta\|_n
\,.
\end{aligned}
\end{equation}
From this we wish to extract an estimate on the error norm $\|x_n^{k+1}-x_{n,\rho}^\delta\|_n$.
However, the operator ${r_n^\delta}''(x_{n,\rho}^\delta)+{r_n^\delta}''(x_n^k)+\lambda^{k+1} I_n$ on the left hand side in \eqref{est1} is positive semidefinite only on the critical cone $C(x_{n,\rho}^\delta)$. 
Thus in the case $\lambda_\rho>0$ in which $C(x_{n,\rho}^\delta)$ only consists of directions orthogonal to $x_{n,\rho}^\delta$, we have to make an additional assumption to avoid negative contributions on the left hand side (that generally might have even larger modulus than the good term $\lambda_\rho\|x_n^{k+1}-x_{n,\rho}^\delta\|_n^2$) 
\begin{equation}\label{posdef_r}
%\forall x_n\in \mathcal{D}(F_n)\setminus\{0\}\,, \|x_n\|_n^2\leq \rho(\delta)\, \ x_n\in  
%\{{r_n^\delta}'(x_n)\}_{\leq} \,, \ r_n^\delta}''(x_n)\mbox{ positive semidefinite on }\, : 
%
%\forall x_{n,\rho}^\delta\in\mbox{argmin}\{ r_n^\delta(x_n)\, : \, x_n\in X_n\,, \ \|x_n\|_n^2\leq \rho\}\, : \ \ 
{r_n^\delta}''(x_{n,\rho}^\delta)+\lambda_\rho I_n \mbox{ is positive definite, where }\lambda_\rho=-\frac{({r_n^\delta}'(x_{n,\rho}^\delta),x_{n,\rho}^\delta)}{\rho}\,.
\end{equation}
Here we have used the fact that the Lagrange multiplier can be explicitely represented due to the necessary optimality conditions \eqref{eq_rho} \eqref{ie_rho}.

Assumption~\eqref{posdef_r} is obviously satisifed if the cost function $r_n^\delta$ is convex, but it also admits nonconvexity possibly arising due to nonlinearity of $F$ and/or $\mathcal{S}$, as the following example shows.

\begin{example}
$n=1$, $\xdag>0$, $C\in(\frac{1}{12{\xdag}^2},\frac{1}{4{\xdag}^2})$, $D> \frac{48C{\xdag}^2-1}{12(1/(2\sqrt{C})-\xdag)^2}>0$
\[
r_1^\delta(x)= \frac12(x-\xdag)^2-C(x-\xdag)^4+D\min\{0,x\}^4\,.
\]
It is easy to see that $x=0$ cannot be a solution to \eqref{minr}, thus according to \eqref{eq_rho}, \eqref{ie_rho}, a solution $x$ has to satisfy 
\begin{eqnarray*}
&&x\in[-\xdag,\xdag]\setminus\{0\}\ \wedge \ \frac{-{r_1^\delta}'(x)}{x}\geq0\\ 
&\Leftrightarrow&  
\Bigl(-\xdag<x<0\ \wedge \ (x-\xdag)(1-4C(x-\xdag)^2)+4Dx^3\geq 0\Bigr)\\
&& \hspace*{1cm}\vee
\Bigl(\xdag>x>0\ \wedge \ (x-\xdag)(1-4C(x-\xdag)^2)\leq 0\Bigr)\\
&\Rightarrow&  
\Bigl(-\xdag<x<0\ \wedge \ 1-4C(x-\xdag)^2\leq 0\Bigr)\vee
\Bigl(\xdag>x>0\ \wedge \ 1-4C(x-\xdag)^2\geq 0\Bigr)\\
&\Leftrightarrow&  
-\xdag<x<\xdag-\frac{1}{2\sqrt{C}}<0\ \vee \
\xdag>x>0
\end{eqnarray*}
where we have used the fact that $C<\frac{1}{4{\xdag}^2}$ in the last equivalence.
Hence we get  
\begin{eqnarray*}
&&x\in[-\xdag,\xdag]\setminus\{0\}\ \wedge \ \frac{-{r_1^\delta}'(x)}{x}\geq0\\ 
&\Rightarrow&  
{r_1^\delta}''(x)+\frac{-{r_1^\delta}'(x)}{x}
= 1-12C(x-\xdag)^2+12 D \min\{0,x\}^2+\frac{-{r_1^\delta}'(x)}{x}\\
&&\left\{\begin{array}{l}
\geq1-48C{\xdag}^2+12D(\frac{1}{2\sqrt{C}}-\xdag)^2>0
\ \mbox{ if }-\xdag<x<\xdag-\frac{1}{2\sqrt{C}}<0\\
=\frac{1}{x}\Bigl(\xdag-4C(2x^3-3x^2\xdag+{\xdag}^3)\Bigr)\geq\frac{\xdag}{x}(1-4C{\xdag}^2)>0
\ \mbox{ if }\xdag>x>0
\end{array}\right.
\end{eqnarray*}
where we have used the fact that $D> \frac{48C{\xdag}^2-1}{12(1/(2\sqrt{C})-\xdag)^2}$ and
$4C{\xdag}^2<1$. 
Thus condition \eqref{posdef_r} is satisfied although $r_1^\delta$ is nonconvex.
An illustration of this example is provided in Figure \ref{fig:ex}.
\begin{figure}
\hspace*{-2cm}
\begin{minipage}{1.3\textwidth}
\includegraphics[width=0.24\textwidth]{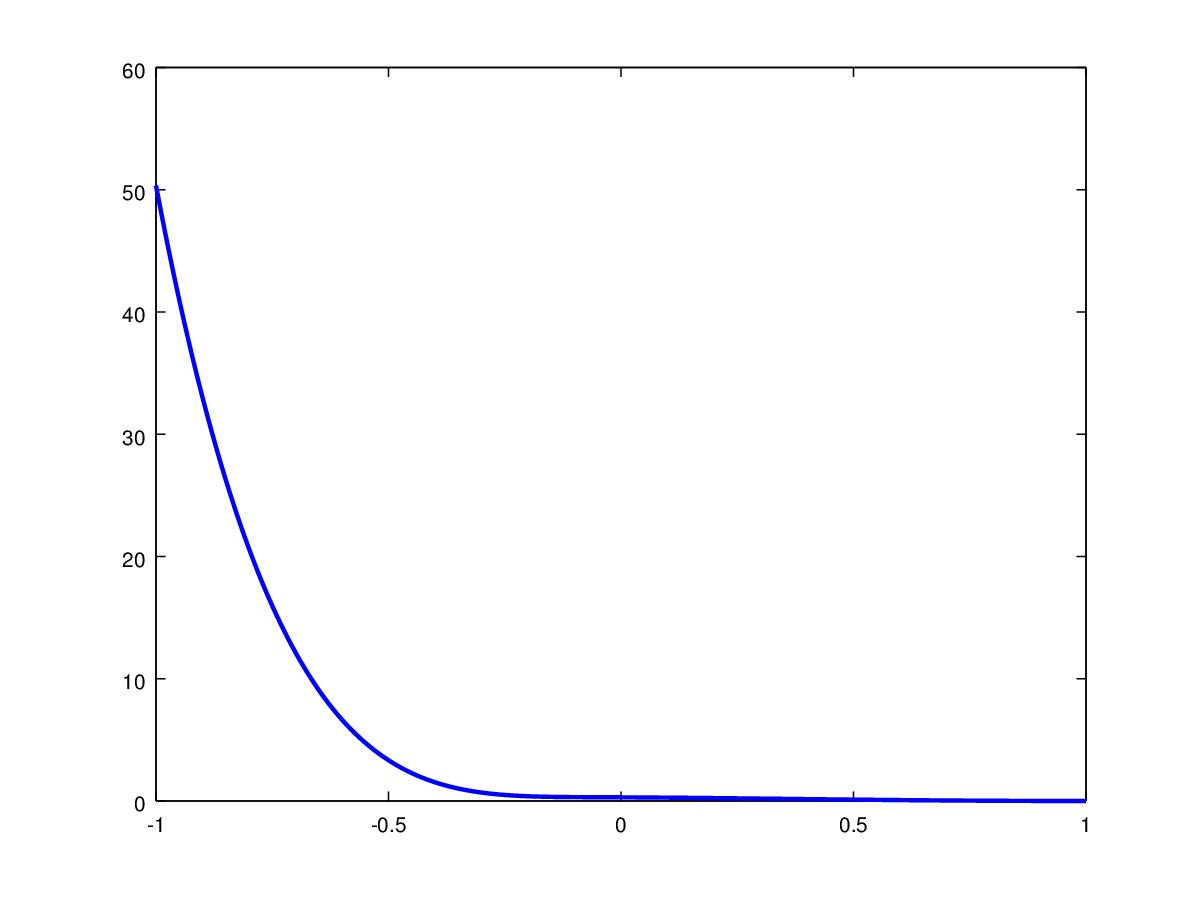}
\includegraphics[width=0.24\textwidth]{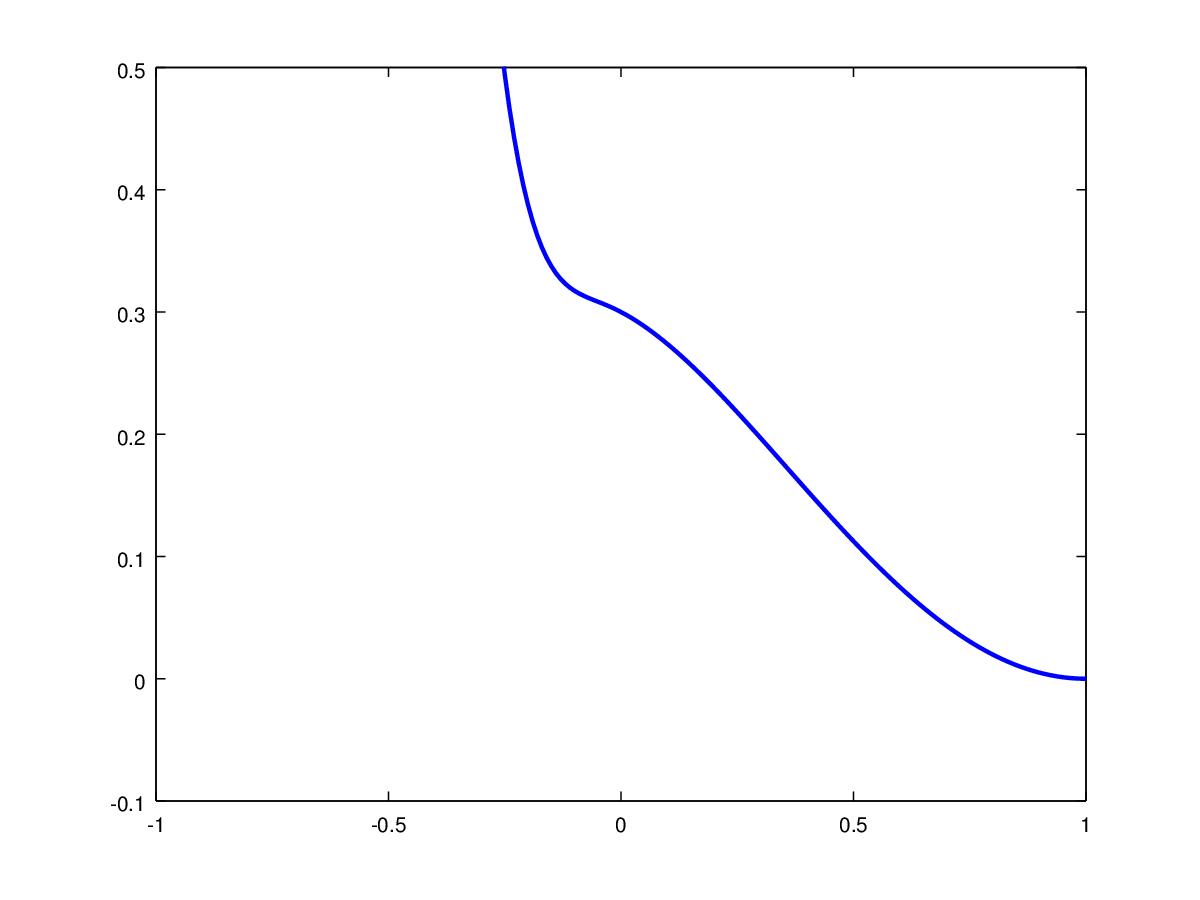}
\includegraphics[width=0.24\textwidth]{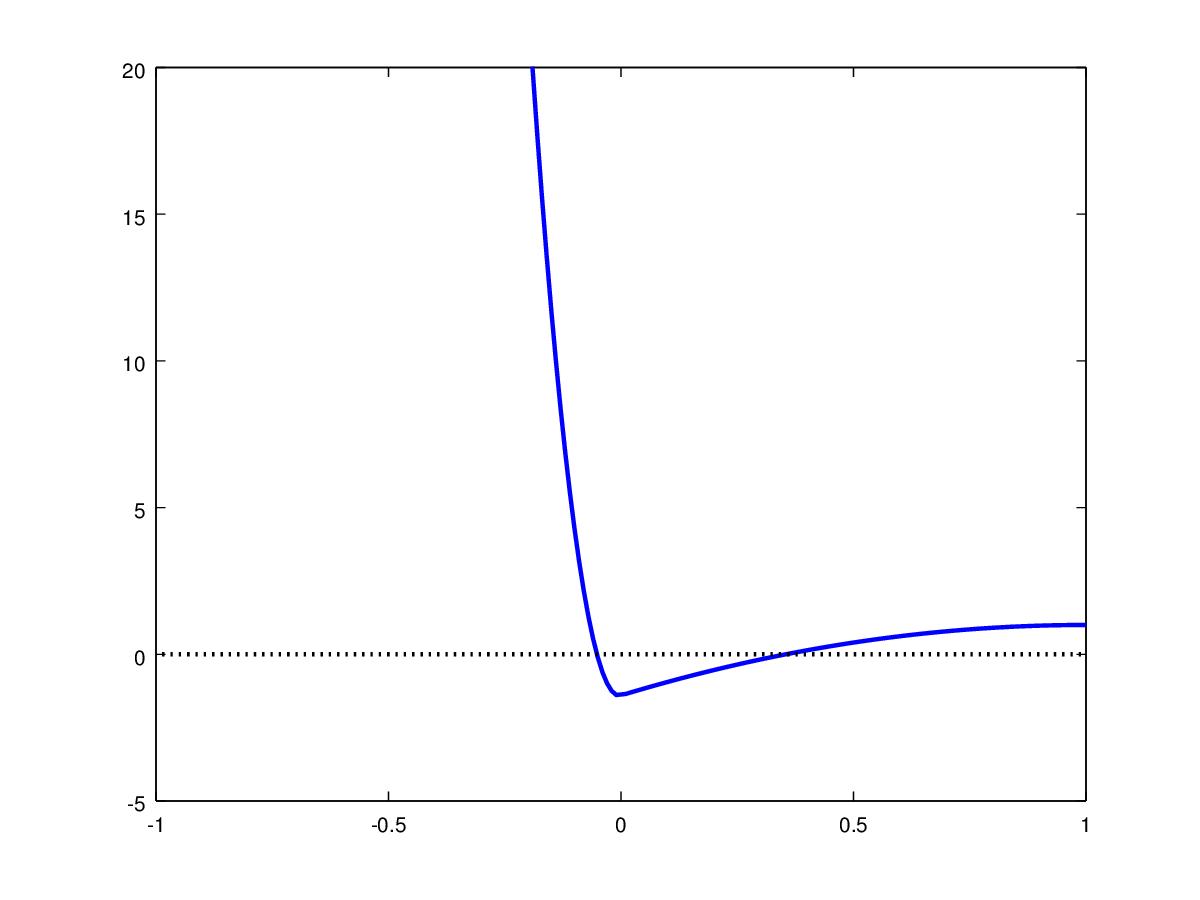}
\includegraphics[width=0.24\textwidth]{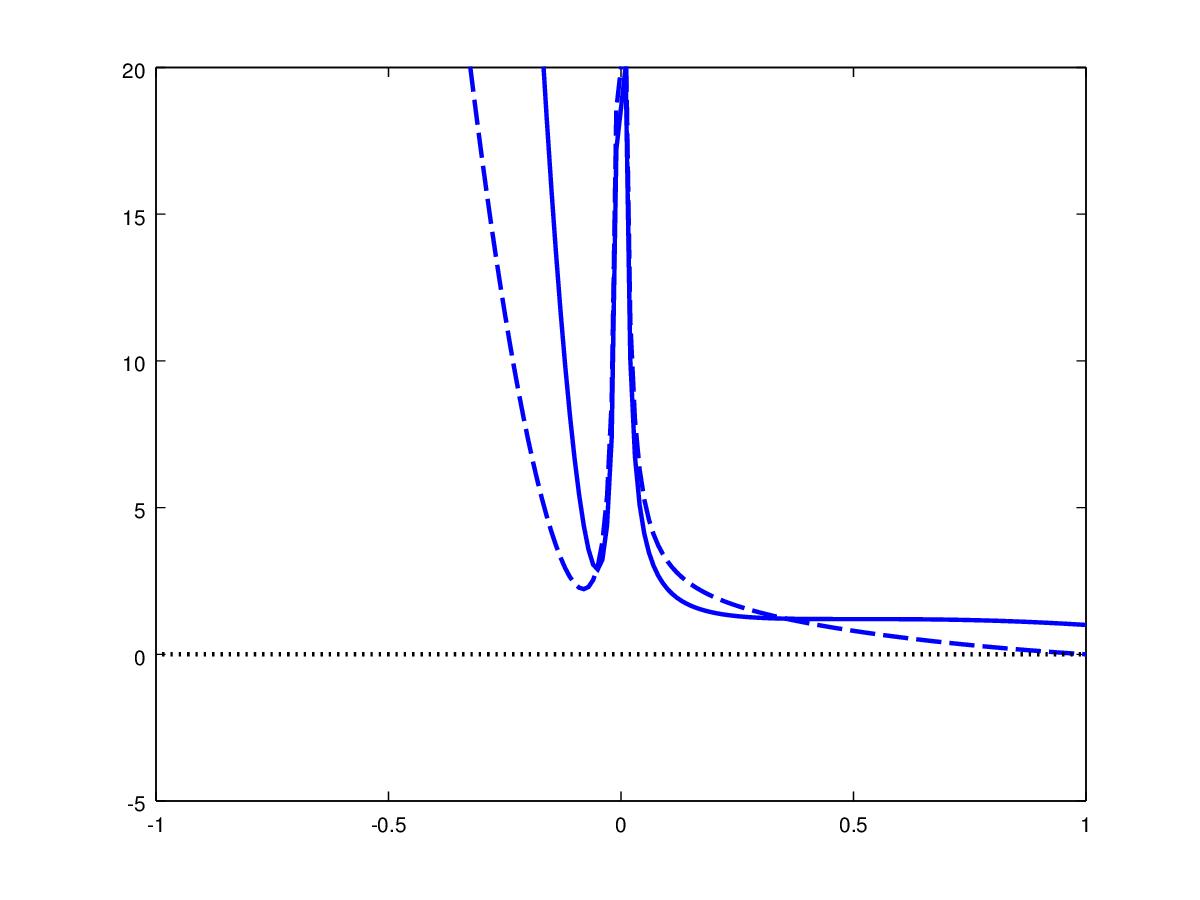}\\
\hspace*{0.09\textwidth}
(a)\hspace*{0.21\textwidth}
(b)\hspace*{0.21\textwidth}
(c)\hspace*{0.21\textwidth}
(d)
\end{minipage}
\caption{
From left to right: (a) function $r_1^\delta$, (b) detailed view of $r_1^\delta$, (c) Hessian ${r_1^\delta}''$, (d) shifted Hessian ${r_1^\delta}''+\lambda_\rho$ (solid) and Lagrange multiplier $\lambda_\rho$ (dashed)
\label{fig:ex}
}
\end{figure}
\end{example}

\begin{proposition}
Let $r_n^\delta$ be twice Lipschitz continuously differentiable \eqref{Lipschitz_rpp} and satisfy \eqref{posdef_r} at some minimizer $x_{n,\rho}^\delta$ of \eqref{minr}.\\
Then the iterates $x^{k+1}$ defined as minimizers to \eqref{minq} converge locally quadratically to $x_{n,\rho}^\delta$.  
\end{proposition}
\begin{proof}
Let $x_{n,\rho}^\delta\in\mbox{argmin}\{ r_n^\delta(x_n) \, : \, x_n\in X_n\,, \ \|x_n\|_n^2\leq \rho\}$.
By finite dimensionality of the space $X_n$, condition \eqref{posdef_r} implies existence of $\alpha_n>0$ such that 
\begin{equation}\label{alphan}
\forall w\in X_n\, : \ 
{r_n^\delta}''(x_{n,\rho}^\delta)w^2-\frac{({r_n^\delta}'(x_{n,\rho}^\delta),x_{n,\rho}^\delta)}{\rho}\|w\|_n^2\geq\alpha_n\|w\|_n^2
\end{equation}
Let the starting point $x^0\in X_n$ be contained in some $\epsilon_n$-neighborhood (wrt. $\|\cdot\|_n$) of $x_{n,\rho}^\delta$ with $\epsilon_n\in(0,\frac{\alpha_n}{3L})$.
Using \eqref{alphan} with $w=x^{k+1}-x_{n,\rho}^\delta$ in \eqref{est1} (recall that ${r_n^\delta}''(x_n^k)+\lambda^{k+1} I_n$ is positive semidefinite on all of $X_n$ by \eqref{sd_kp1}) for $k=0$ by the triangle inequality yields
\[
\alpha_n \|x_n^{k+1}-x_{n,\rho}^\delta\|_n^2
\leq L\bigl(2\|x_{n,\rho}^\delta-x_n^k\|_n^2+\|x_{n,\rho}^\delta-x_n^k\|_n\, \|x_n^{k+1}-x_{n,\rho}^\delta\|_n\bigr)\|x_n^{k+1}-x_{n,\rho}^\delta\|_n\,,
\]
hence
\begin{eqnarray}
\|x_n^{k+1}-x_{n,\rho}^\delta\|_n
&\leq& \frac{2L}{\alpha_n -L\epsilon_n}\|x_{n,\rho}^\delta-x_n^k\|_n^2\nonumber\\
&\leq& \frac{3L}{\alpha_n}\|x_{n,\rho}^\delta-x_n^k\|_n^2 \label{quad}\\
&\leq& \frac{3L\epsilon_n}{\alpha_n}\|x_{n,\rho}^\delta-x_n^k\|_n\label{contr}\\
&<&\epsilon_n\,.\label{neighb}
\end{eqnarray}
By an inductive argument using the same estimate for general $k$, the iterates remain in the $\epsilon_n$-neighborhood of $x_{n,\rho}^\delta$ (cf. \eqref{neighb}) and satisfy a contraction (cf. \eqref{contr}) as well as a quadratic convergence (cf. \eqref{quad}) estimate.
\end{proof}

\medskip

We now consider these conditions in more detail in the special case of a Hilbert space norm 
\begin{equation}\label{Snorm}
\mathcal{S}(y_1,y_2)=\frac12\|y_1-y_2\|^2
\end{equation}
for measuring the data discrepancy. The optimality conditions \eqref{eq_rho}--\eqref{sd_rho} then in terms of the discretized forward operator $F_n$ read as follows:
\begin{equation}\label{optsys_F}
\begin{aligned}
&(a)
\begin{cases}
\lambda_\rho=0
\mbox{ and }
\|x_{n,\rho}^\delta\|_n^2\leq\rho\mbox{ and }\\
\forall w\in X_n \, : \ \langle F_n(x_{n,\rho}^\delta)-\ydel,F_n'(x_{n,\rho}^\delta)w\rangle_n=0
\mbox{ and }\\
\forall w\in X_n \, : \ 
\| F_n'(x_{n,\rho}^\delta)w\|_n^2
+\langle F_n(x_{n,\rho}^\delta)-\ydel,F_n''(x_{n,\rho}^\delta)w^2\rangle_n
\geq0\\
\end{cases}\\
&\mbox{ or }\\
&(b)\begin{cases} 
\lambda_\rho=\frac{-\langle F_n(x_{n,\rho}^\delta)-\ydel,F_n'(x_{n,\rho}^\delta)x_{n,\rho}^\delta\rangle_n}{\rho}>0\mbox{ and }
\|x_{n,\rho}^\delta\|_n^2=\rho\mbox{ and }\\
\forall w\in \{x_{n,\rho}^\delta\}^\bot \, : \ 
\langle F_n(x_{n,\rho}^\delta)-\ydel,F_n'(x_{n,\rho}^\delta)w\rangle_n
=0\mbox{ and }\\
\forall w\in \{x_{n,\rho}^\delta\}^\bot \, : \ 
\| F_n'(x_{n,\rho}^\delta)w\|_n^2
+\langle F_n(x_{n,\rho}^\delta)-\ydel,F_n''(x_{n,\rho}^\delta)w^2\rangle_n
\geq0
\end{cases}
\end{aligned}
\end{equation}
where 
in case (a) we use the fact that for all $w\in X_n$ either $w$ or $-w$ is in the critical cone and 
in case (b) we have used the fact that for any $w\in X_n$
\[
\langle F_n(x_{n,\rho}^\delta)-\ydel,F_n'(x_{n,\rho}^\delta)w\rangle_n+\lambda_\rho
\langle x_{n,\rho}^\delta,w\rangle_n=
\langle F_n(x_{n,\rho}^\delta)-\ydel,F_n'(x_{n,\rho}^\delta)\mbox{Proj}_{\{x_{n,\rho}^\delta\}^\bot}w\rangle_n
\]
As a matter of fact this shows that the case (a) of a vanishing Lagrange multiplier is not really relevant here, since with 
$\hat{Y}_n:=\overline{\mathcal{R}(F_n'(x_{n,\rho}^\delta))}$ 
it implies that $\mbox{Proj}_{\hat{Y}_n}(F_n(x_{n,\rho}^\delta)-\ydel)=0$ so that 
if $\overline{\bigcup_{n\in\N} \hat{Y}_n}=Y$ we get
\[
r_n^\delta(x_{n,\rho}^\delta)=\frac12\|F_n(x_{n,\rho}^\delta)-\ydel\|_n^2=
\frac12\|\mbox{Proj}_{\hat{Y}_n^\bot}(F_n(x_{n,\rho}^\delta)-\ydel)\|_n^2
<\delta+\underline{\eta}_n^\delta-\eta_n^\delta\,,
\]
where the last inequality holds for $n$ sufficiently large, provided $\delta+\underline{\eta}_n^\delta-\eta_n^\delta>0$, which is compatible with the assumptions made in Section \ref{sec:reg}.
On the other hand, from \eqref{Ivanov_n_eta} and \eqref{discrprinc_n_eta} it follows that 
\[
r_n^\delta(x_{n,\rho(\delta)}^\delta)>\delta+\underline{\eta}_n^\delta-\eta_n^\delta
\]
which by Lemma~\ref{lem:mon} implies $\rho>\rho(\delta)$.\\
The positivity condition in \eqref{posdef_r} becomes 
\begin{equation}\label{posdef_F}
\forall w\in X_n\setminus\{0\} \, : \ \|F_n'(x_{n,\rho}^\delta)w\|_n^2+\langle F_n(x_{n,\rho}^\delta)-\ydel, F_n''(x_{n,\rho}^\delta)w^2\rangle_n +\lambda_\rho\|w\|^2>0\,.
\end{equation}
This condition will indeed be satisfied for $\rho=\rho(\delta)$ for $\delta$ sufficiently small and $n$ sufficiently large, e.g., in the sitution of an estimate 
\[
\forall w\in X_n\, : \ \|F_n''(x_{n,\rho}^\delta)w^2\|_n\leq C \|F_n'(x_{n,\rho}^\delta)w\|_n^2
\]
(which might be interpreted as a condition on the nonlinearity of the forward operator)
holding with a constant $C$ independent of $n$, since then for all $w\in X_n\setminus\{0\}$
\[
\begin{aligned} 
&\|F_n'(x_{n,\rho}^\delta)w\|_n^2+\langle F_n(x_{n,\rho}^\delta)-\ydel, F_n''(x_{n,\rho}^\delta)w^2\rangle_n +\lambda_\rho\|w\|^2\\
&\geq\|F_n'(x_{n,\rho}^\delta)w\|_n^2-(\tau\delta+\overline{\eta}_n^\delta) \|F_n''(x_{n,\rho}^\delta)w^2\|_n +\lambda_\rho\|w\|^2\\
&\geq(1-(\tau\delta+\overline{\eta}_n^\delta)C)\|F_n'(x_{n,\rho}^\delta)w\|_n^2+\lambda_\rho\|w\|^2>0\,.
\end{aligned}
\]
However, Assumption \eqref{posdef_F} possibly remains valid also in case $x_{n,\rho}^\delta$ is still far away from $\xdag$ since then the positive contribution of the (then typically larger) Lagrange multiplier takes effect: Note that both the residual $F_n(x_{n,\rho}^\delta)-\ydel$ and the norm $\frac{1}{\sqrt{\rho}}$ of the ratio $\frac{x_{n,\rho}^\delta}{\rho}$ get larger for smaller $\rho<\|\xdag\|^2$. 

\begin{corollary}
Let $Y$ be a Hilbert space, let $\mathcal{S}$, $\mathcal{S}_n$ be defined by the squared norm \eqref{Snorm} and its finite dimensional approximation $\mathcal{S}_n(y_1,y_2)=\frac12\|y_1-y_2\|_n^2$, assume that $F_n$ is twice Lipschitz continuously differentiable and that for some minimizer  $x_{n,\rho}^\delta$ of \eqref{minr} condition \eqref{posdef_F}  holds. \\
Then the iterates $x_n^{k+1}$ defined as minimizers to \eqref{minq} converge locally quadratically to $x_{n,\rho}^\delta$ .  
\end{corollary}

\begin{remark}
In view of the well-known equivalence between the Levenberg Marquardt method and the application of a trust region method to successive quadratic approximations of the nonlinear least squares cost function, there is an obvious relation to \cite{HankeLevMar}, still more, since we also use the discrepancy principle for choosing the trust region radius (as is done for the regularization parameter $\alpha$ in \cite{HankeLevMar}). The main difference to the method described here, besides the somewhat more general data space setting, lies in the fact that we start from the nonlinear trust region problem and work with quadratic approximations of the cost function, that are possibly nonconvex. This is why the algorithm from \cite{RendlWolkowicz97} plays a key role here.
\end{remark}
\subsection{Solving the nonconvex quadratic trust region subproblem}
Discretization $x_n=\sum_{j=1}^n x^i\phi^i$ with a basis $\{\phi^1,\ldots,\phi^n\}$ of $X_n$ leads to an equivalent formulation of \eqref{minq} as a constrained quadratic minimization problem
\begin{equation}\label{TRS}
\min_{x\in \R^n} \ x^TAx-2a^Tx \mbox{ s.t. } x^Tx\leq s^2
\end{equation}
with a not necessarily positive semidefinite Hessian $A$.
Such problems can be very efficiently solved by the method proposed in \cite{RendlWolkowicz97}, which we briefly sketch in the following.

The key idea relies in the fact that solving \eqref{TRS} can be recast into a root finding problem for a parametrized eigenvalue problem. This can be motivated by the following chain of inequalities, where $\mu^*$ is the optimal value of \eqref{TRS} (see Section 2.2 in \cite{RendlWolkowicz97}).
\begin{eqnarray}
\mu^*&=&\min_{\|x\|=s, y_0^2=1} x^TAx-2y_0a^Tx
\nonumber\\
&=&\max_t \min_{\|x\|=s, y_0^2=1} x^TAx-2y_0a^Tx+ty_0^2-t
\nonumber\\
&\geq&\max_t \min_{\|x\|^2+y_0^2=s^2+1} x^TAx-2y_0a^Tx+ty_0^2-t
\label{maxmineigval}\\
&\geq&\max_{t,\lambda} \min_{x,y_0} x^TAx-2y_0a^Tx+ty_0^2-t+\lambda(\|x\|^2+y_0^2-s^2-1)
\nonumber\\
&=&\max_{r,\lambda} \min_{x,y_0} x^TAx-2y_0a^Tx+ry_0^2-r+\lambda(\|x\|^2-s^2)
\nonumber\\
&=&\max_{\lambda} \Bigl(\max_{r} \min_{x,y_0} x^TAx-2y_0a^Tx+ry_0^2-r+\lambda(\|x\|^2-s^2)\Bigr)
\nonumber\\
&=&\max_{\lambda} \min_x \min_{y_0^2=1} x^TAx-2y_0a^Tx+\lambda(\|x\|^2-s^2)\ = \mu^*
\nonumber
\end{eqnarray}
Indeed, with $y=\left(\begin{array}{c}y_0\\x\end{array}\right)$, the optimization problem on line \eqref{maxmineigval} can be rewritten as
\begin{equation}\label{k} 
\max_t\min_{\|y\|^2=s^2+1}y^T D(t) y-t=\max_t (s^2+1)\lambda_{\min}(D(t))-t = \max_t k(t)
\end{equation}
where 
\begin{equation}\label{Dt}
D(t)=\left(\begin{array}{cc} t&-a\\-a&A\end{array}\right)
\end{equation} 
and $\lambda_{\min}(M)$ denotes the smallest (possibly negative) eigenvalue of some matrix $M$.

More precisely, it can be shown (Theorem 14 in \cite{RendlWolkowicz97}), that unless the so-called hard case occurs, for any $t\in\R$ and for any $\left(\begin{array}{c}y_0(t)\\z(t)\end{array}\right)$ being a normalized eigenvector corresponding to $\lambda_{\min}(D(t))$, the vector $x=\frac{1}{y_0(t)}z(t)$ is well-defined and solves \eqref{TRS} with $s=\frac{1-y_0(t)^2}{y_0(t)^2}$
(Here the ``hard case'' is the pathological situation of $a$ being orthogonal to the eigenspace corresponding to $\lambda_{\min}(A)$.) 

Based on this observation, it remains to iteratively find a root of the function $\psi(t)=\sqrt{s^2+1}-\frac{1}{y_0(t)}$ (that can be shown to be almost linear, nonincreasing and concave), or equivalently, to find a stationary point of the function $k$ in \eqref{k}, which can be done very efficiently using inverse interpolation, cf. \cite{RendlWolkowicz97}.

The main computational effort of the resulting algorithm lies in the determination of an eigenvector corresponding to the smallest eigenvalue of $D(t)$ in each of these root finding iterations. For this purpose fast routines exist, (in our numerical tests we use the code from \url{http://www.math.uwaterloo.ca/~hwolkowi//henry/software/trustreg.d/} employing the Matlab routine eigs based on an Arnoldi method).
Since this method only uses matrix vector products with $A$, it suffices to provide a routine for evaluating the linear operators $F_n'(x)$, $F_n''(x)^*(F_n(x)-\ydel)$ in a given direction $d$, which particularly makes sense for $F_n$ being the (discretization of a) forward operator for some inverse problem, e.g., some parameter identification problem in a PDE. Namely, in that case these actions just correspond to solving the underlying linearized PDE model with some inhomogeneity involving $d$.

\subsection{Newton's method for computing $\rho(\delta)$}

In view of the discrepancy principle \eqref{discrprinc_n_eta} for choosing $\rho=\rho(\delta)$, we have to approximate a root of the one-dimensional function $\phi$ defined by 
\begin{equation}\label{phi}
\phi(\rho)=r_n^\delta(x_{n,\rho}^\delta)-r_d \quad \mbox{ where }r_d=\frac{(\tau+1)\delta+\overline{\eta}_n^\delta+\underline{\eta}_n^\delta}{2}\,,
\end{equation}
which by Lemma~\ref{lem:mon} is monotonically decreasing. For this purpose, Newton's method is known to converge globally and quadratically, provided $\phi$ is twice continuously differentiable. As a matter of fact, in the generic case of $x_{n,\rho}^\delta$ lying on the boundary of the feasible set, the derivative of $\phi$ is just the Lagrange multiplier for \eqref{minr}, since by the complementarity condition in \eqref{ie_rho}
\[
\begin{aligned}
\phi'(\rho)=&\frac{d}{d\rho} r_n^\delta(x_{n,\rho}^\delta)=
\frac{d}{d\rho} \Bigl(r_n^\delta(x_{n,\rho}^\delta)+\lambda_\rho(\|x_{n,\rho}^\delta\|^2-\rho)\Bigr)\\
=& {r_n^\delta}'(x_{n,\rho}^\delta)\frac{d x_{n,\rho}^\delta}{d\rho}
+\frac{d\lambda_\rho}{d\rho}(\|x_{n,\rho}^\delta\|^2-\rho)
+\lambda_\rho \langle x_{n,\rho}^\delta,\frac{d x_{n,\rho}^\delta}{d\rho}\rangle_n
-\lambda_\rho\\
=&-\lambda_\rho
\end{aligned}
\]
by \eqref{eq_rho} and $\|x_{n,\rho}^\delta\|^2=\rho$.
A similar observation has already been made for the derivative of the cost functional with respect to the regularization parameter in Tikhonov regularization cf. \cite{KKV11}.

\subsection{Algorithm}
Altogether we arrive at the following nested iteration.
 
\begin{algorithm}{.}\label{alg}
\begin{algorithmic}[1]
\STATE Given $\delta>0$, choose $\tau>1$, $n$, 
		$\underline{\eta}_n^\delta$, $\overline{\eta}_n^\delta$, $\eta_n^\delta$ 
		such that \eqref{eta1}, \eqref{eta2}, \eqref{eta3} are satisfied, \\$\rho_0=0$, $\rho_1>0$;
\STATE Set $l=1$, $\hat{x}_{n,\rho_0}^\delta=0$, $\hat{x}_{n,\rho_1}^\delta=0$;
\WHILE
{$r_n^\delta(\hat{x}_{n,\rho_l}^\delta) > \tau\delta +\overline{\eta}_n^\delta$}
%{$\cdots$ (Iteration for $\rho$)} 
	\STATE $k=0$, $x_n^0=\hat{x}_{n,\rho_{l-1}}^\delta$
	\WHILE
{$r_n^\delta(x_n^k)>\min_{x_n\in X_n} \{r_n^\delta(x_n)\ : \ \|x_n\|_n^2\leq \rho_l\}+\eta_n^\delta$}
%{$\cdots$ (Iteration for $\hat{x}_{n,\rho}^\delta$)}  
		\STATE Compute solution $x_n^{k+1}$ and Lagrange multiplier $\lambda^{k+1}$ to \eqref{minq}, \eqref{q}, $\rho=\rho_l$  
		\STATE Set $k=k+1$ 
    \ENDWHILE    
	\STATE Set $\hat{x}_{n,\rho_l}=x_n^k$, $\lambda_{\rho_l}=\lambda^k$
	\STATE Set $\rho_{l+1}=\rho_l+\frac{r_n^\delta(\hat{x}_{n,\rho_l}^\delta)-r_d}{\lambda_{\rho_l}}$ with $r_d$ as in \eqref{phi}
	 \STATE Set $l=l+1$ 
\ENDWHILE    
 \end{algorithmic}
\end{algorithm}

\section{Numerical tests}

To illustrate performance of Algorithm \ref{alg} we make use of an implementation of the method described in \cite{RendlWolkowicz97} available on the web page of one of the authors \url{http://www.math.uwaterloo.ca/~hwolkowi//henry/software/trustreg.d/} and consider the nonlinear integral equation
\[
y(t)=\int_0^t (x(s)+10 x(s)^3)\, ds \ t\in(0,1)
\]
with $X=Y=L^2(0,1)$ and $\mathcal{S}(y_1,y_2)=\frac12\|y_1-y_2\|_{L^2}^2$.
The integral equation is discretized with a composite trapeziodal rule on an equidistant grid with $100$ breakpoints. % and the zero function is used as a starting value in all our tests.

Figure \ref{fig_data_recons} shows the noisy data as well as the exact and reconstructed solutions with different noise levels.
\begin{figure}
%\hspace*{-2cm}
%\begin{minipage}{1.3\textwidth}
\includegraphics[width=0.32\textwidth]{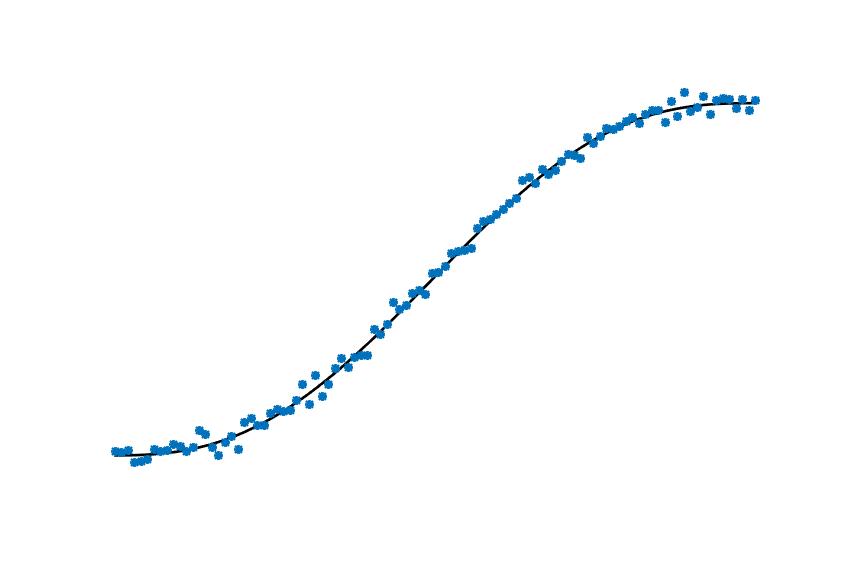}
\includegraphics[width=0.32\textwidth]{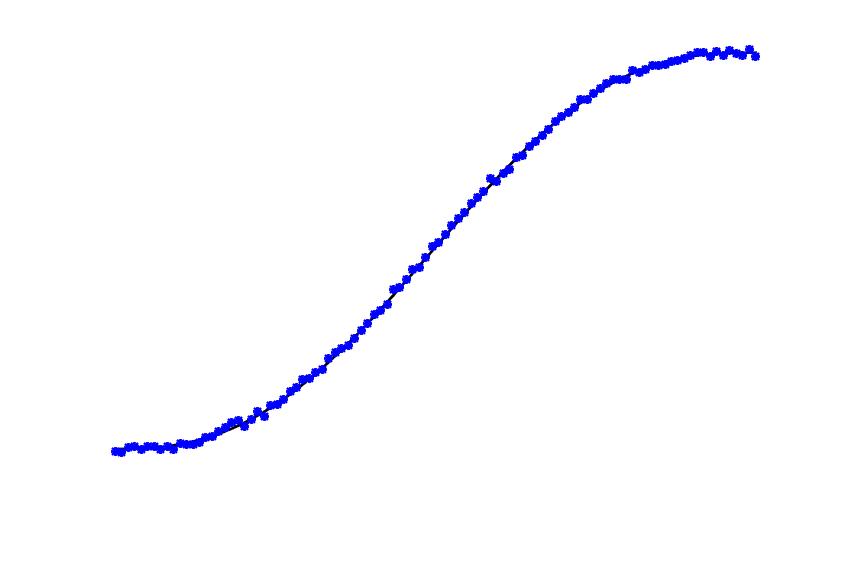}
\includegraphics[width=0.32\textwidth]{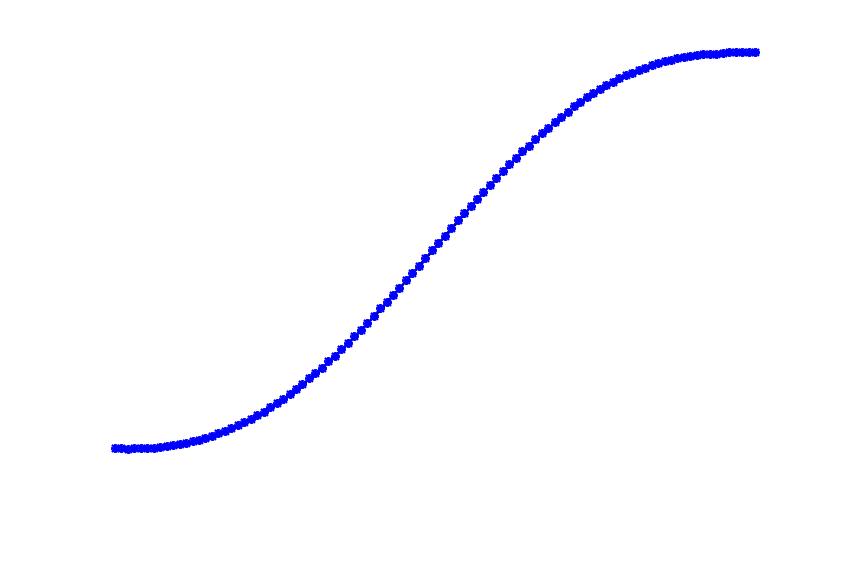}
\\
\includegraphics[width=0.32\textwidth]{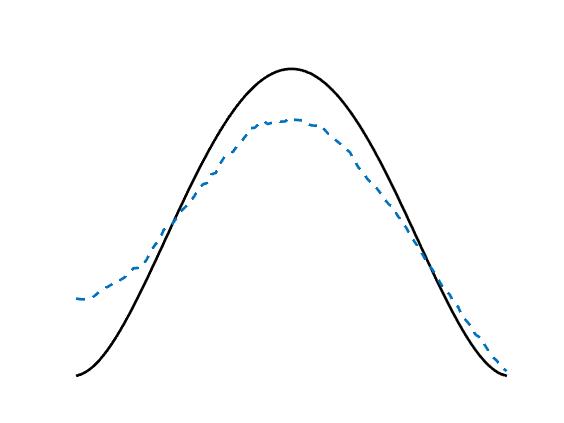}
\includegraphics[width=0.32\textwidth]{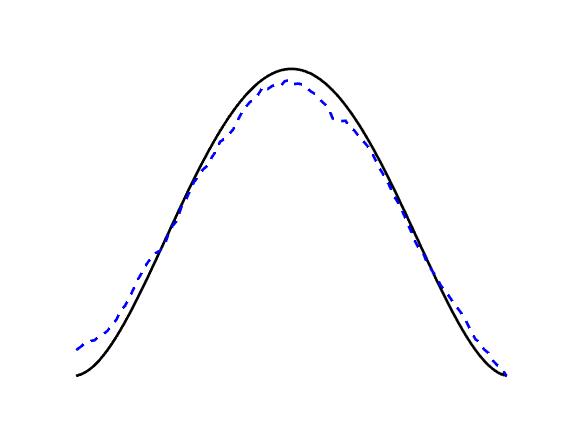}
\includegraphics[width=0.32\textwidth]{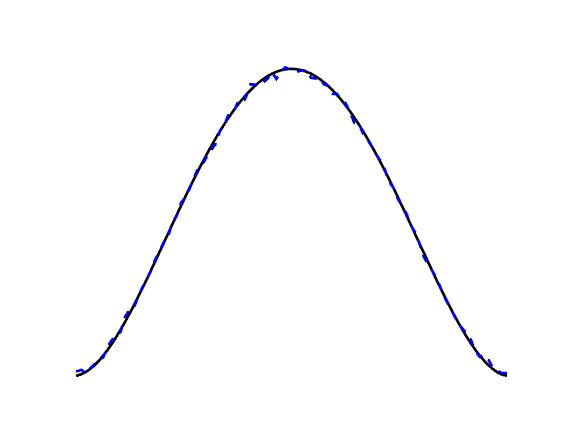}
\\
\hspace*{0.07\textwidth}
$\delta=3\%$	\hspace*{0.21\textwidth}
$\delta=1\%$	\hspace*{0.21\textwidth}
$\delta=0.1\%$	\hspace*{0.21\textwidth}
%$\delta=0$	
%\\
%\end{minipage}
\caption{Noisy data (top row) and reconstructions (bottom row) for $\delta=3\%$ (left column), $\delta=1\%$ (middle column), $\delta=0.1\%$ (right column)
\label{fig_data_recons} }
\end{figure}
In figure \ref{fig_convhist} we plot error and residual, as well as the number of solved nonconvex quadratic subproblems over the radius, for noise levels $1$ and $3$ per cent, during the iteration over $\rho$.
\begin{figure}
\includegraphics[width=0.49\textwidth]{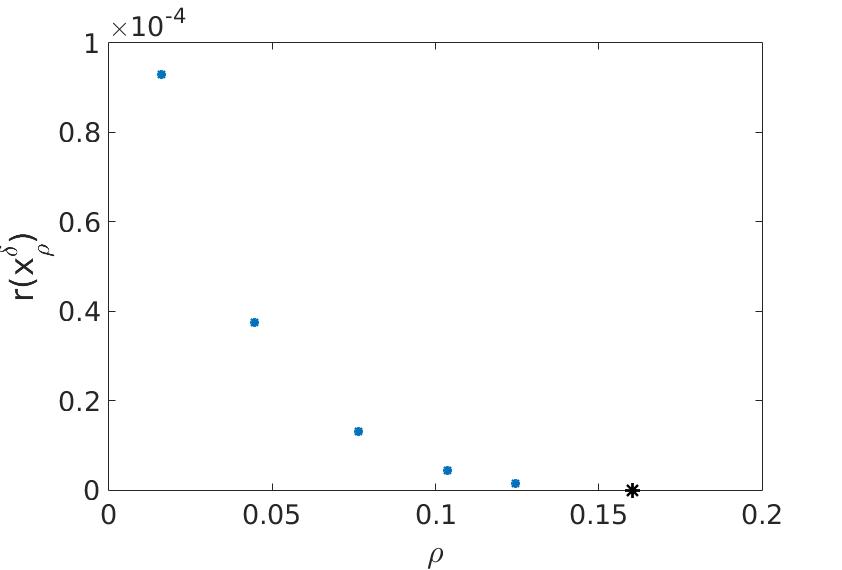}
\includegraphics[width=0.49\textwidth]{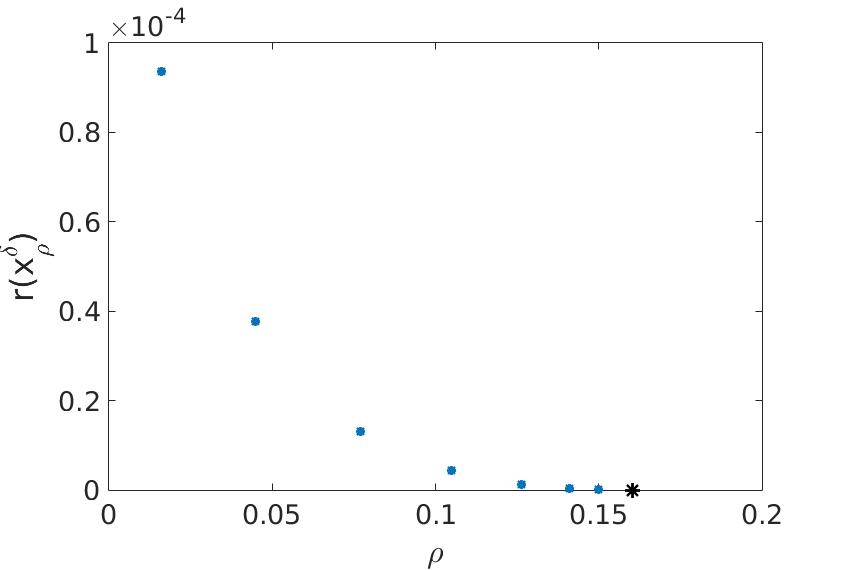}\\
\includegraphics[width=0.49\textwidth]{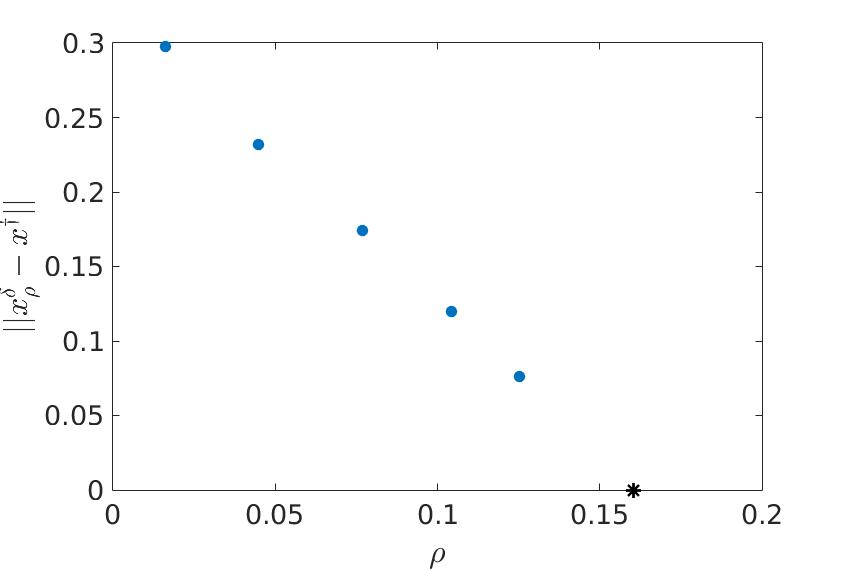}
\includegraphics[width=0.49\textwidth]{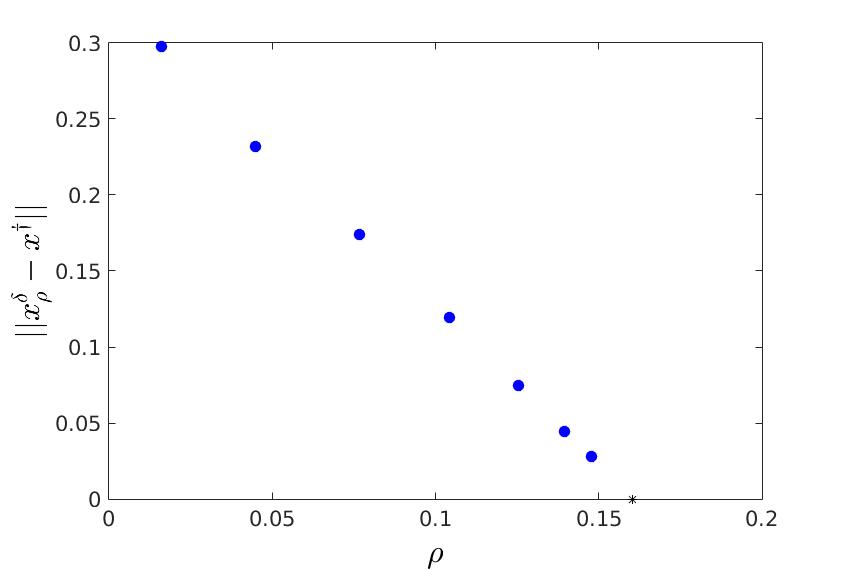}
\caption{Residuals (top row) and errors (bottom row) for $\delta=3\%$ (left column) and $\delta=1\%$ (right column).
\label{fig_convhist} }
\end{figure}
Figure \ref{fig_convhist1} illustrates that condition \eqref{posdef_r} remains valid throughout the iteration over $\rho$, while several nonconvex trust region subproblems have to be solved (the number being smaller for  large $\delta$ since less iterations are carried out in that case).
\begin{figure}
\includegraphics[width=0.49\textwidth]{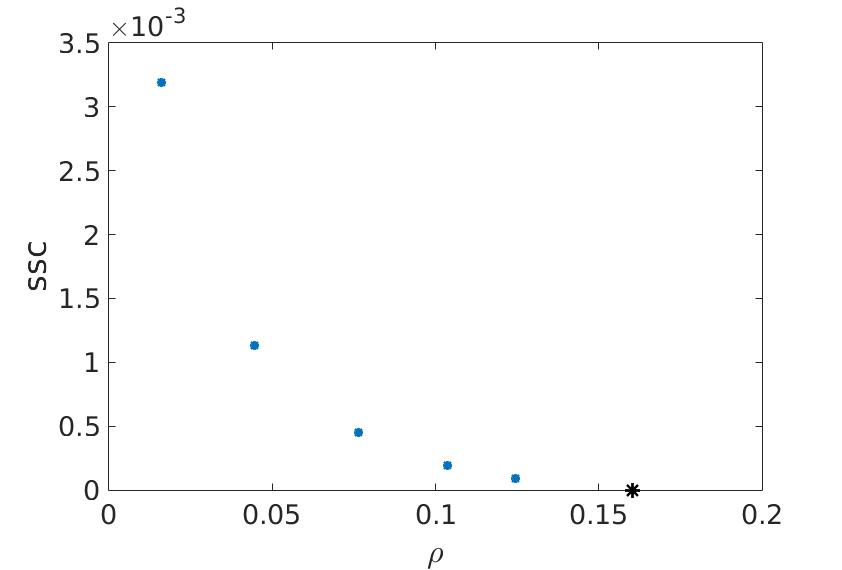}
\includegraphics[width=0.49\textwidth]{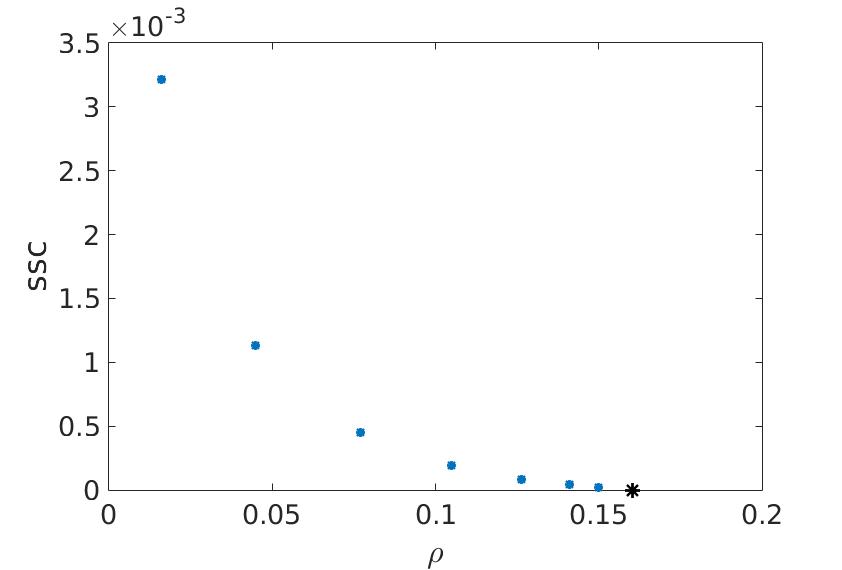}\\
\includegraphics[width=0.49\textwidth]{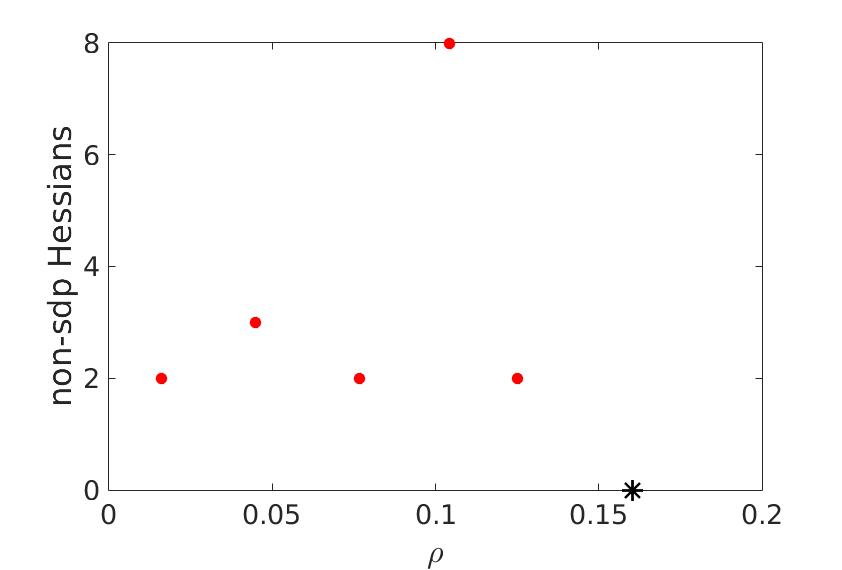}
\includegraphics[width=0.49\textwidth]{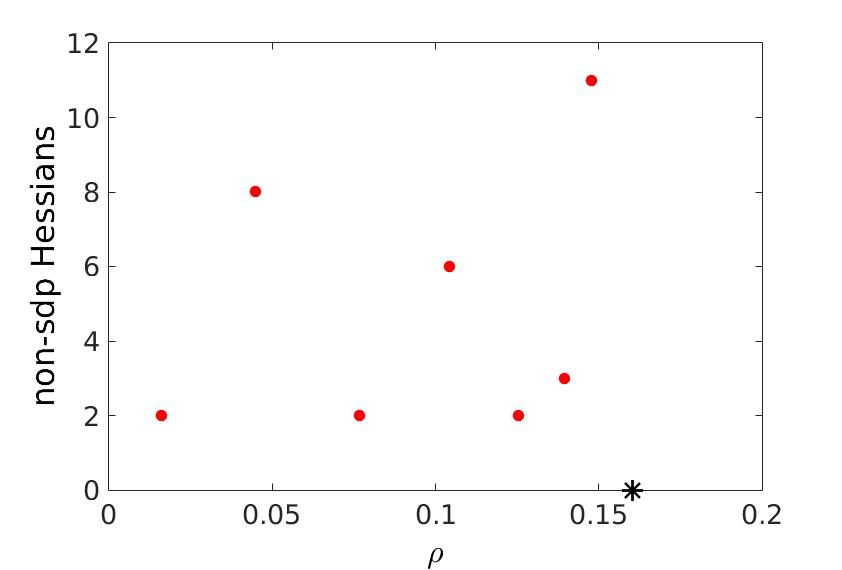}
\caption{Smallest eigenvalue of Hessian of Lagrangian (top row) and number of nonconvex quadratic subproblems (bottom row) for $\delta=3\%$ (left column) and $\delta=1\%$ (right column).
\label{fig_convhist1} }
\end{figure}
%As a real world example, consider identification of the magnetic reluctivity in nonlinear magnetics \cite{KaKaRe03}.

\bibliographystyle{siam}
\bibliography{lit_IvanovTRS}

\begin{thebibliography}{10}

\bibitem{CKK15}
{\sc C.~Clason, B.~Kaltenbacher, and A.~Klassen}, {\em {On convergence and
  convergence rates for Ivanov and Morozov regularization}},  (2015).
\newblock in preparation.

\bibitem{GrodzevichWolkowicz09}
{\sc O.~Grodzevich and H.~Wolkowicz}, {\em Regularization using a parameterized
  trust region subproblem}, Math. Program., Ser. B, 116 (2009), pp.~193--220.

\bibitem{HankeLevMar}
{\sc M.~Hanke}, {\em {A regularization Levenberg--Marquardt scheme, with
  applications to inverse groundwater filtration problems}}, Inverse Problems,
  13 (1997), pp.~79--95.

\bibitem{Ivanov62}
{\sc V.~K. Ivanov}, {\em On linear problems which are not well-posed}, Dokl.
  Akad. Nauk SSSR, 145 (1962), pp.~270--272.

\bibitem{Ivanov63}
\leavevmode\vrule height 2pt depth -1.6pt width 23pt, {\em On ill-posed
  problems}, Mat. Sb. (N.S.), 61 (103) (1963), pp.~211--223.

\bibitem{IvanovVasinTanana02}
{\sc V.~K. Ivanov, V.~V. Vasin, and V.~P. Tanana}, {\em Theory of Linear
  Ill-posed Problems and Its Applications}, Inverse and ill-posed problems
  series, VSP, 2002.

\bibitem{KKV11}
{\sc B.~Kaltenbacher, A.~Kirchner, and B.~Vexler}, {\em Adaptive
  discretizations for the choice of a {T}ikhonov regularization parameter in
  nonlinear inverse problems}, Inverse Problems, 27 (2011), p.~125008.

\bibitem{LorenzWorliczek13}
{\sc D.~Lorenz and N.~Worliczek}, {\em Necessary conditions for variational
  regularization schemes}, Inverse Problems, 29 (2013), p.~075016.

\bibitem{NeubauerRamlau14}
{\sc A.~Neubauer and R.~Ramlau}, {\em {On convergence rates for quasi-solutions
  of ill-posed problems.}}, ETNA, Electron. Trans. Numer. Anal., 41 (2014),
  pp.~81--92.

\bibitem{RendlWolkowicz97}
{\sc F.~Rendl and H.~Wolkowicz}, {\em A semidefinite framework for trust region
  subproblems with applications to large scale minimization}, Math. Program.,
  77 (1997), pp.~273--299.

\bibitem{SeidmanVogel89}
{\sc T.~I. Seidman and C.~R. Vogel}, {\em Well posedness and convergence of
  some regularisation methods for non-linear ill posed problems}, Inverse
  Problems, 5 (1989), p.~227.

\bibitem{Sorensen82}
{\sc D.~Sorensen}, {\em Newton’s method with a model trust region
  modification}, SIAM J. Numer.Anal., 19 (1982), pp.~409--426.

\bibitem{Vogel90}
{\sc C.~R. Vogel}, {\em A constrained least squares regularization method for
  nonlinear iii-posed problems}, SIAM Journal on Control and Optimization, 28
  (1990), pp.~34--49.

\end{thebibliography}

\end{document}